\documentclass[letterpaper, 10 pt, conference]{ieeeconf}  

\IEEEoverridecommandlockouts                              

\usepackage{amsmath}
\usepackage{amssymb}
\usepackage{algorithm}
\usepackage{algpseudocode}

\usepackage[noadjust]{cite}

\newcounter{tempEquationCounter}
\newcounter{thisEquationNumber}
\newenvironment{floatEq}
{\setcounter{thisEquationNumber}{\value{equation}}\addtocounter{equation}{1}
\begin{figure*}[!t]
\normalsize\setcounter{tempEquationCounter}{\value{equation}}
\setcounter{equation}{\value{thisEquationNumber}}
}
{\setcounter{equation}{\value{tempEquationCounter}}
\hrulefill\vspace*{4pt}
\end{figure*}

}

\newtheorem{theorem}{Theorem}
\newtheorem{definition}{Definition}
\newtheorem{proposition}{Proposition}

\title{\LARGE \bf
Family-Personalized Dietary Planning with Temporal Dynamics
}

\author{Pedro Hespanhol and Anil Aswani
\thanks{*This work was supported by the Hellman Fellows Program.}
\thanks{Pedro Hespanhol and Anil Aswani are with the Department of Industrial Engineering and Operations Research, University of California, Berkeley, CA 94720, USA 
        {\tt\small pedrohespanhol@berkeley.edu, aaswani@berkeley.edu}}%
}

\begin{document}

\maketitle
\thispagestyle{empty}
\pagestyle{empty}

\begin{abstract}
Poor diet and nutrition in the United States has immense financial and health costs, and development of new tools for diet planning could help families better balance their financial and temporal constraints with the quality of their diet and meals.  This paper formulates a novel model for dietary planning that incorporates two types of temporal constraints (i.e., dynamics on the perishability of raw ingredients over time, and constraints on the time required to prepare meals) by explicitly incorporating the relationship between raw ingredients and selected food recipes.  Our formulation is a diet planning model with integer-valued decision variables, and so we study the problem of designing approximation algorithms (i.e, algorithms with polynomial-time computation and guarantees on the quality of the computed solution) for our dietary model.  We develop a deterministic approximation algorithm that is based on a deterministic variant of  randomized rounding, and then evaluate our deterministic approximation algorithm with numerical experiments of dietary planning using a database of about 2000 food recipes and 150 raw ingredients.
\end{abstract}

\section{Introduction}

Poor diet and nutrition in the United States costs an estimated \$700 billion per year \cite{heidenreich2011forecasting,american2013economic} due to increases in diseases like type 2 diabetes and cardiovascular disease.  Diet quality is also important for managing body weight \cite{dpp2002,fukuoka2011}.  Given the importance of diet in maintaining good health, clinically-supervised programs \cite{dpp2003,dpp2009,fukuoka2014b} provide nutritional counseling to encourage participants to improve their diet.  

Since such counseling is costly, clinicians are studying how the sensing, computation, and communication capabilities of mobile devices can be integrated into the design of clinically-supervised programs in order to reduce costs \cite{fukuoka2011,fukuoka2014b}.  More recently, adaptive control \cite{aswani2016,mintz2017} has been used to personalize the physical activity goals and scheduling of counseling sessions in weight loss programs.  However, the control problem of designing personalized dietary plans that consider the temporal constraints imposed by ingredient purchasing and perishability has been less well-studied.


\subsection{Dietary Planning}

Diet planning was one of the first optimization problems to be formulated \cite{stigler1945}.  Existing formulations have focused on the problem of selecting a set of raw food ingredients subject to a financial budgetary constraint and bounds on the nutrients of the selected ingredients.  The earliest formulations focused on linear programs (LP's) in which continuous quantities of ingredients are selected \cite{stigler1945}.  More recent formulations (including those used for governmental policy decision-making) focus on minimizing convex functions of continuous quantities of ingredients \cite{carlson2007thrifty} or selecting discrete (i.e., integer-valued) quantities of menu items \cite{lancaster1992history}.  

However, a substantial weakness \cite{rose2007food} of these formulations is they do not include constraints for time required to prepare meals from the raw ingredients.  Furthermore, these formulations do not consider that raw ingredients leftover from a previous time period could be used to prepare meals in the current time period.  One contribution of this paper is to formulate a new model for dietary planning that includes these two types of temporal constraints: constraints for the amount of time to prepare meals, and constraints to describe how raw ingredients can be used over multiple time periods. 

Our formulation for dietary planning includes the two types of temporal constraints by explicitly including the joint choices of deciding which raw ingredients to purchase at each time period and which recipes/meals to prepare at each time period. And the goal is to choose these two sets of integer-valued quantities in order to maximize the quality of the selected meal plans.  This paper does not study how to estimate preferences, but instead assumes that preferences are already known; however, in principle inverse optimization or other learning-based approaches \cite{aswani2016,aswani2015inverse,aswani2017statistics,aswani2013provably,mintz2016bilevel} could potentially be used to estimate the preferences of meal plans with different sets of raw ingredients and meals.

\subsection{Approximation Algorithms for Integer Packing}

Because the model for designing dietary plans involves integer optimization, numerical solution requires development of algorithms that can scale to large time horizons and large numbers of instances (for each family).  Approximation algorithms for integer optimization provide some possible insights.  Let $U\in[0,1]^{d_1\times m}$, $u\in[1,\infty)^{d_1}$, and $c\in[0,1]^m$ with $\|c\|_\infty = 1$.  Then a packing integer problem (PIP) is
\begin{equation}
\label{eqn:pip}
\max \{c^\mathsf{T}x\ |\ Ux\leq u, x\in\mathbb{Z}^m_+\}.
\end{equation}

Approximation algorithms (i.e., polynomial-time computation with a bound on the suboptimality of computed solutions) based on randomized rounding \cite{raghavan1987randomized} or pessimistic estimators \cite{raghavan1988probabilistic,srinivasan1999improved} have been developed for PIP's.  Unfortunately, these algorithms cannot handle the constraints $Ax_1 \leq Bx_2$, where matrices $A,B$ have nonnegative entries and $x_1,x_2$ is a partition of the decision variable $x$, which is necessary to constrain the relationship between selected food recipes and purchased raw ingredients.  A second contribution of this paper is to develop approximation algorithms using randomized rounding and pessimistic estimators for a more general formulation with constraints of the form $Ax_1\leq Bx_2$.

\subsection{Outline}

We first present a new formulation of dietary planning with temporal dynamics.  This formulation includes the problem of selecting both raw ingredients and food recipes, and in this way allows inclusion of two temporal constraints that limit the time required to prepare meals and capture the dynamics of perishability of raw ingredients over time.  This formulation is a diet planning model that involves integer optimization, and we define an abstract optimization problem we call a generalized packing integer program (GPIP) that includes our model as a special case.  We construct a randomized approximation algorithm to solve GPIP, and then we extend this algorithm in order to construct a deterministic approximation algorithm.  The deterministic approximation algorithm provides solutions of the same quality as the randomized approximation algorithm; but whereas the randomized algorithm does not always return a feasible solution, the deterministic algorithm always returns a feasible solution.  (This is a general feature of comparison between deterministic and randomized approximation algorithms \cite{raghavan1988probabilistic,srinivasan1999improved}.)  Finally, we conclude with a simulation study to evaluate the computational scaling and solution quality of dietary plans produced by our approximation algorithms.




\section{Dietary Planning and Packing Problems}

This section first describes our approach for dietary planning with temporal constraints to model the perishability of ingredients. By changing the coefficients in this formulation, our dietary plans can be personalized to accommodate different food preferences and dietary restrictions.  Next, we describe an abstract problem that we call a generalized packing integer program (GPIP), and we briefly explain how our dietary planning problem is a special case of GPIP.

\subsection{Dietary Planning Model with Temporal Dynamics}
Let $[r] = \{1,\ldots,r\}$.  We propose performing dietary planning by solving the following optimization problem:
\begin{equation}
\label{eqn:dietmodel}
\begin{aligned}
\max_{x_n,y_n}\ &\textstyle\sum_{n=1}^{N}v^\mathsf{T}x_n + w^\mathsf{T}y_n\\
\text{s.t. } &z_n = y_n + y_{n-1} - Px_{n-1},&&\text{for }n\in[N]\\
&Px_n\leq z_n ,&&\text{for }n\in[N]\\
&Fx_n \leq h,&&\text{for }n\in[N]\\
&\textstyle\sum_{n=1}^{N}{x_{n,r}} \leq f_r, &&\text{for } r\in[R] \\ 
&t^\mathsf{T}x_n \leq T,\  b^\mathsf{T}y_n \leq B,&&\text{for }n\in[N]\\
&x_n \in \mathbb{Z}^{m}, y_n \in \mathbb{Z}^{k},z_n\in\mathbb{R}^k &&\text{for }n\in[N]
\end{aligned}
\end{equation}

The intuition of this optimization problem is as follows: The $x_{n,r} \in\mathbb{Z}$ denotes the quantity of recipe $r$ selected at time period $n$.  Similarly, the $y_{n,i}\in\mathbb{Z}$ indicates the number of packages of ingredient $i$ purchased at the $n$-th time period, while the $z_{n,i}\in\mathbb{R}$ are the portions of packages $i$ available for cooking at the $n$-th time period. The goal is to select recipes and ingredients that maximize a linear utility function.

The dynamics $z_n = y_n + y_{n-1} - Px_{n-1}$ say the portion of packages at $n$ is equal to the the number of packages purchased at $n$ plus the portion of packages remaining from the last time period $n-1$.  This model incorporates the notion of perishability of ingredients, and for simplicity we assume ingredients expire after two time periods; however, these dynamics could be suitably modified to model that different ingredients will have different time horizons of perishability.

The constraint $Px_n\leq z_n$ ensures sufficient portions of ingredient packages are available to prepare the recipes that have been selected, while $Fx_n \leq h$ ensures that appropriate nutrition (e.g., calories, vitamins, fat content) is obtained from the chosen recipes.  The $\textstyle\sum_{n=1}^{N}{x_{n,r}} \leq f_r$ inequalities place a limit on the number of times particular recipes are selected over the entire planning horizon $N$.  (Note that $f_r = 0$ ensures that no amount of recipe $r$ is selected.)  The $t^\mathsf{T}x_n \leq T$ inequality constrains the total time to prepare all the recipes at $n$ to be within the time budget $T$, and the $b^\mathsf{T}y_n \leq B$ inequality ensures that the total cost of ingredient packages purchased at $n$ is less than a financial budget $B$.

\subsection{Generalized Packing Integer Program (GPIP)}

Next we describe a general class of optimization problems.  Let $A\in[0,1]^{n\times m}$, $B\in[0,1]^{n\times k}$, $U\in[0,1]^{d_1\times m}$, $V\in[0,1]^{d_2\times k}$, $u\in[1,\infty)^{d_1}$, $v\in[1,\infty)^{d_2}$, $c_1\in[0,1]^m$, and $c_2\in[0,1]^k$ with $\|c_1\|_\infty = 1$ and $\|c_2\|_\infty = 1$.  Then we define a general packing integer problem (GPIP) as
\begin{equation}
\label{eqn:gpip}
\begin{aligned}
\max\ &c_1^\mathsf{T}x + c_2^\mathsf{T}y\\
\text{s.t. } & Ax\leq By\\
&Ux \leq u,\, Vy\leq v\\
&x\in\mathbb{Z}^m_+, y\in\mathbb{Z}^k_+
\end{aligned}
\end{equation}
This is closely related to our dietary planning model (\ref{eqn:dietmodel}) with temporal constraints, since we can replace $z_n$ in the constraint $Px_n\leq z_n$ by its dynamics $z_n = y_n + y_{n-1} - Px_{n-1}$; this leads to the GPIP structure after rearranging the terms of the resulting inequality.  Though GPIP only contains integer variables, our approximation algorithms generalize naturally to the case where some variables in GPIP are continuous.

\section{Randomized Algorithm for GPIP}

This section designs a randomized approximation algorithm to solve the GPIP problem.  To simplify the exposition, we will assume without loss of generality that the decision variables in GPIP are binary: $x\in\{0,1\}^m$ and $y\in\{0,1\}^k$.  Recall that $(\hat{x} , \hat{y})$ is a solution to the LP relaxation of (\ref{eqn:gpip}) if it solves the modified optimization problem that consists of (\ref{eqn:gpip}) but with the last constraints replaced with $x\in\mathbb{R}^m_+, y\in\mathbb{R}^k_+$.  Our general approach (similar to the approach of \cite{raghavan1987randomized} in approximating PIP) is to first solve the LP relaxation of (\ref{eqn:gpip}) and then strategically round this solution. Algorithm \ref{alg:randound} summarizes our randomized approximation algorithm.

The key technical challenge is finding an appropriate rounding strategy that allows us to bound the quality of the resulting solution.  In order to round the solution, we will construct two random vectors $X$ and $Y$ such that a single sample from these two vectors provides a good solution to GPIP.  Define $(x',y') =(\hat{x}/\alpha , \hat{y}/\gamma)$ where $p\gamma = \alpha$, $p >1$, and $\gamma > 1 $. Let $X_i\in\{0,1\}^m$ be a vector of independent Bernoulli random variables where the success probability of the $i$-th component is $x'_i$. Similarly, let $Y_i\in\{0,1\}^k$ be a vector of independent Bernoulli random variables where the success probability of the $i$-th component is $y'_i$.  

\subsection{Deviation Bounds for Single Events}

We consider undesirable events that correspond to constraints being violated by the randomized solution or to the objective function value with the rounded solution being small.  Let the first subscript on a matrix be the row, and let the second subscript be the column.  For example $A_i$ is the $i$-th row of matrix $A$, while $A_{ij}$ is the $ij$-th entry of $A$.  With this notation, undesirable events are given by
\begin{equation}
\begin{aligned}
&E_i:= ({A_i^\mathsf{T}X > B_i^\mathsf{T}Y })\\
&Q_{i}:= (U_i^\mathsf{T}X > \mu^{1}_{i}(1+ \delta^{1}_{i}))\\
&R_{i}:= (V_i^\mathsf{T}Y > \mu^{2}_{i}(1+ \delta^{2}_{i}))\\
&E_{n+1}:= (c^\mathsf{T}_{1}X + c^\mathsf{T}_{2}Y < \mu^{n+1}(1-\delta^{n+1}))
\end{aligned}
\end{equation}
where for a constant $\beta$ we have that
\begin{equation}
\begin{aligned}
&\mu^{0}_{i}= \mathbb{E}(A_i^\mathsf{T}X)\\
&\mu^{1}_{i} =\mathbb{E}(U_i^\mathsf{T}X) & \quad&\delta^{1}_{i} =u_i/\mu^{1}_{i}  - 1\\
&\mu^{2}_{i} =\mathbb{E}(V_i^\mathsf{T}Y) & \quad&\delta^{2}_{i} =v_i/\mu^{2}_{i}  - 1\\
&\mu^{n+1} =\mathbb{E}(c^\mathsf{T}_{1}X + c^\mathsf{T}_{2}Y)&\quad&\delta^{n+1} = \frac{c^\mathsf{T}_{1}\hat{x} + c^\mathsf{T}_{2}\hat{y}}{\alpha\beta\mu^{n+1}}
\end{aligned}
\end{equation}
Without loss of generality, we assume $\mu^{0}_{i}$, $\mu^{1}_{i}$, $\mu^{2}_{i}$ are strictly positive because we can eliminate any decision variables $x_i$ with $x_i' = 0$ or $y_i$ with $y_i' =0$ by setting them to zero and then considering GPIP with those variables eliminated.

Let $(x_n)_{\min} = \min\{x_{n,i}\ |\ x_{n,i} > 0\}$.  Our first step is to quantify the likelihood of undesirable events occurring.  
\begin{proposition}
\label{prop:singbound}
We have the following probability bounds
\begin{equation}
\begin{aligned}
&\mathbb{P}(Q_i) \leq G(u_i/\alpha , \alpha - 1), &&\forall i \in [d_1]\\
&\mathbb{P}(R_i) \leq G(v_i/\gamma , \gamma - 1), &&\forall i \in [d_2]\\
&\mathbb{P}(E_{n+1}) \leq H(z^{*}/\alpha, 1- 1/\beta)\\
&\mathbb{P}(E_i) \leq \mathbb{P}(B^\mathsf{T}_{i}Y = 0)\cdot\mathbb{P}(A^\mathsf{T}_{i}X > 0) + \\
&\hspace{0.5cm}\mathbb{P}(B^\mathsf{T}_{i}Y > 0)\cdot G( (B_i)_{\min}/\alpha , \alpha-1 ), &&\forall i\in[n]
\end{aligned}
\end{equation}
where we have that $G(\mu,\delta) = (\exp(\delta)/(1+\delta)^{(1+\delta)})^\mu$ and $H(\mu,\delta)  = \exp(-\mu\delta^{2}/2)$.
\end{proposition}

\begin{proof}
The first three inequalities follow by combining the Chernoff-Hoeffding bound \cite{hoeffding1963probability,raghavan1988probabilistic} with the inequalities from \cite{srinivasan1999improved} that:
\begin{equation}
\begin{aligned}
& G(\mu^{1}_{i},\delta^{1}_{i}) \leq G(u_i/\alpha, \alpha - 1), &\forall i \in [d_1]\\
& G(\mu^{2}_{i},\delta^{2}_{i}) \leq G(v_i/\gamma, \gamma - 1), &\forall i \in [d_2]\\
& H(\mu^{n+1},\delta^{n+1}) \leq H(z^{*}/\alpha, 1- 1/\beta).  
\end{aligned}
\end{equation}

Proving the fourth inequality requires additional work.  We condition on whether or not the random variable $B_{i}^\mathsf{T}Y$ is equal to the zero.  If $B_{i}^\mathsf{T}Y>0$ and $A^\mathsf{T}_{i}X > B^\mathsf{T}_{i}Y$, then $A_i^\mathsf{T}X$ is bigger than $(B_i)_{\min}$. Hence we get the bound
\begin{multline}
\mathbb{P}(E_i) \leq \mathbb{P}(B_{i}^\mathsf{T}Y=0)\cdot\mathbb{P}(A^\mathsf{T}_{i}X > 0) + \\
\mathbb{P}(B_{i}^\mathsf{T}Y > 0)\cdot\mathbb{P}(A^\mathsf{T}_{i}X > (B_i)_{\min} ).
\end{multline}
Next define $\delta^{0}_{i} = (B_i)_{\min}/( \mu^{0}_{i}) - 1$. If $(B_i)_{\min} \geq \mathbb{E}(B_i^\mathsf{T}Y)$, then $\delta^{0}_{i} \geq 0$ and we can use the Chernoff-Hoeffding bound \cite{raghavan1988probabilistic,hoeffding1963probability}. On the other hand, we need to ensure that $(B_i)_{\min} \geq \mu^{0}_{i}$ in order to get a deviation of the random variable $A_i^\mathsf{T}X$ above its mean. This requires $\alpha$ be multiplied by the constant factor of $\|A_i\|_1/(B_i)_{\min}$.  And so we have:
\begin{multline}
 \mathbb{P}(A^\mathsf{T}_{i}X > (B_i)_{\min} ) \leq \\
\mathbb{P}(A_i^\mathsf{T}X >\mu^{0}_{i}(1+\delta^{0}_{i} )) \leq G( \mu^{0}_{i} ,\delta^{0}_{i} )
\end{multline}
Since $\mu^{0}_{i} \leq (B_i)_{\min} / p$ and $\alpha > 1$, the above bounds from \cite{srinivasan1999improved} give the fourth inequality.
\end{proof}

\subsection{Deviation Bound for Union of Events}

\begin{algorithm}[t]
\begin{algorithmic}[1]
\Require Constants $\alpha,\gamma$
\Require LP Relaxation Solution $\hat{x},\hat{y}$
\State choose $x_i = 1$ (resp., $x_i = 0$) with probability $\hat{x}_i/\alpha$ (resp., with probability $1-\hat{x}_i/\alpha$)
\State choose $y_i = 1$ (resp., $y_i = 0$) with probability $\hat{y}_i/\gamma$ (resp., with probability $1-\hat{y}_i/\gamma$)
\State \Return $(x,y)$
\end{algorithmic}
\caption{Randomized Rounding Algorithm for GPIP\label{alg:randound}}
\end{algorithm}

To prove that Algorithm \ref{alg:randound} is an approximation algorithm, we need to next quantify the likelihood of the above described undesirable events occurring.  The following proposition provides needed bounds for unions of undesirable events:

\begin{proposition}
\label{prop:unbnd}
If $\alpha=\Omega(m/k + (n+d_1)^{1/([B, u])_{\min}})$, $\beta =1-\sqrt{2}/\sqrt{3}$, and $\gamma = \Omega(d_2^{1/(v)_{\min}})$; then we have that
\begin{equation}
\begin{aligned}
&\textstyle\mathbb{P}(\bigcup_{i=1}^{d_1}Q_{i}) < 1/5&\quad&\textstyle\mathbb{P}(\bigcup_{i=1}^{d_2}R_{i}) < 1/5\\
&\textstyle\mathbb{P}(\bigcup_{i=1}^{n}E_{i}) < 1/5&\quad&\textstyle\mathbb{P}(E_{n+1}) < 2/5
\end{aligned}
\end{equation}
whenever $(c^\mathsf{T}_{1}\hat{x} + c^\mathsf{T}_{2}\hat{y})/\alpha > 5$.
\end{proposition}

\begin{proof}
Let $x(S)$ be such that $x(S)_{i} = 1$ if and only if $i \in S$, and let $y(T)$ be such that $y(T)_{j} = 1$ if and only if $i \notin T$. Next define the sets
\begin{equation}
\begin{aligned}
& \mathcal{F}^0_i=\{S \subseteq [n], T \subseteq [k] : A_i^\mathsf{T}x(S) \leq B_i^\mathsf{T}y(T) \} \\
& \mathcal{F}^1_i=\{S \subseteq [n] : U_i^\mathsf{T}x(S) \leq\mu^{1}_{i}(1+ \delta_i^1)\}\\
& \mathcal{F}^2_i=\{T \subseteq [k] : V_i^\mathsf{T}y(S) \leq\mu^{2}_{i}(1+ \delta_i^2)\}
\end{aligned}
\end{equation}
The $\mathcal{F}^{0}_i$ and $\mathcal{F}^{1}_i$ are monotone decreasing, while the $\mathcal{F}_i^{2}$ are monotone increasing. (A set $\mathcal{F}$ is monotone increasing if $S\subseteq T$ with $S\in\mathcal{F}$ implies $T\in\mathcal{F}$, and $\mathcal{F}$ is monotone decreasing if $S\subseteq T$ with $T\in\mathcal{F}$ implies $S\in\mathcal{F}$.) 

Hence the Fortuin-Kasteleyn-Ginibre (FKG) inequality \cite{fortuin1971correlation} gives
\begin{equation}
\begin{aligned}
&\textstyle \mathbb{P}(\bigcup_{i=1}^{d_1}Q_{i}) \leq 1- \prod_{i=1}^{d_1}(1-\mathbb{P}(Q_{i})) \\
&\textstyle \mathbb{P}(\bigcup_{i=1}^{d_2}R_{i}) \leq 1- \prod_{i=1}^{d_2}(1-\mathbb{P}(R_{i}))
\end{aligned}
\end{equation}
Proposition \ref{prop:singbound} implies we have $\mathbb{P}(\bigcup_{i=1}^{d_1}Q_{i}) < 1/5$ whenever $1-(1-\exp((u)_{\min} - (u)_{\min}\log(\alpha)))^{d_1} <  1/5$.  If $\alpha \geq 3$, then $\log(\alpha) - 1 >0$ and there exists $k'>0$ such that
\begin{multline}
1-\exp((u)_{\min} - (u)_{\min}\log(\alpha))) \geq \\
\exp(-k'\exp(-(u)_{\min}\log(\alpha)-1)).
\end{multline}
So we require the two inequalities:
\begin{equation}
\begin{aligned}
& \exp(k'\exp(-(u)_{\min}\log(\alpha-1)) > \sqrt[d_1]{4/5} \\
& -k'\exp((u)_{\min})\alpha^{-(u)_{\min}} > \log(\sqrt[d_1]{4/5})
\end{aligned}
\end{equation}
Let $K_2 =k'\exp((u)_{\min})$ and $K_3=-K_2/\log(4/5)$, and note that $K_2, K_3>0$. We have that
\begin{equation}
\begin{aligned}
 \alpha > -(K_2/\log{\sqrt[d_1]{4/5}})^{1/(u)_{\min}},
\end{aligned}
\end{equation}
 and so we require that $\alpha > (K_3d_1)^{1/(u)_{\min}}$ in order to ensure $\mathbb{P}(\bigcup_{i=1}^{d_1}Q_{i}) < 1/5$.  The same argument shows $\mathbb{P}(\bigcup_{i=1}^{d_2}R_{i}) < 1/5$ when we have that $\gamma > p(K_4d_2)^{1/(v)_{\min}}$ for a constant $K_4 >0$.

We next study $\mathbb{P}(\bigcup_{i=1}^{n}E_{i})$.  Note we can decompose these events as: $E_i = E_i^{1} \bigcup E_i^{2}$, where
\begin{equation}
\begin{aligned}
& E_i^{1} := ({A^\mathsf{T}_{i}X > B^\mathsf{T}_{i}Y } \wedge  B^\mathsf{T}_{i}Y=0) \\
& E_i^{2} := ({A^\mathsf{T}_{i}X > B^\mathsf{T}_{i}Y } \wedge  B^\mathsf{T}_{i}Y>0)
\end{aligned}
\end{equation}
The union bound gives
\begin{equation}
\textstyle \mathbb{P}(\bigcup_{i=1}^{n}E_i)  \leq \mathbb{P}(\bigcup_{i=1}^{n}E_i^{1}) + \mathbb{P}(\bigcup_{i=1}^{n}E_i^{1}),
\end{equation}
and so $\mathbb{P}(\bigcup_{i=1}^{n}E_{i}) < 1/5$ whenever $\mathbb{P}(\bigcup_{i=1}^{n}E_i^{1}) < 1/10$ and $\mathbb{P}(\bigcup_{i=1}^{n}E_i^{2}) < 1/10$.

Applying the FKG inequality means we need
\begin{equation}
\begin{aligned}
&\textstyle 1-\prod_{i=1}^{n}(1-\mathbb{P}(E_i^{1})) <  1/10\\
&\textstyle 1-\prod_{i=1}^{n}(1-\mathbb{P}(E_i^{2})) <  1/10.
\end{aligned}
\end{equation}
But note
\begin{multline}
\mathbb{P}(E_i^{1}) = \mathbb{P}(B^\mathsf{T}_{i}Y=0)\cdot\mathbb{P}(A^\mathsf{T}_{i}X > 0) = \\
\textstyle\mathbb{P}(B^\mathsf{T}_{i}Y=0)\cdot(1-\prod_{j : A_{ij} >0}(1 - x'_{j}/\alpha))
\end{multline}
For any $\gamma > 1$ we have
\begin{equation}
\textstyle \max_i\mathbb{P}(B^\mathsf{T}_{i}Y=0) = \max_i(\prod_{j : B_{ij} > 0} (1-y'_j)) \leq K_5^k
\end{equation}
where $K_5 := \max_j (1-y'_j)$ is a constant.  We also have
\begin{equation}
\textstyle (1-\prod_{j : A_{ij} >0}(1 - x'_{j}/\alpha)) \leq 1 - (1 - \|x'\|_\infty/\alpha)^{m}
\end{equation}
So $\mathbb{P}(E_i^1) \leq K_5^k\cdot (1 - (1 - \|x'\|_\infty/\alpha)^{m})$, and we require
\begin{multline}
\textstyle 1-\prod_{i=1}^{n}(1-\mathbb{P}(E_i^{1})) \leq \\
\textstyle 1 - (1-K_5^k\cdot(1 -(1 - \|x'\|_\infty/\alpha)^{m}))^n <1/10
\end{multline}
or equivalently that
\begin{equation}
\begin{aligned}
 (1 - \|x'\|_\infty/\alpha)^{m} > (K_5^k -1 + \sqrt[n]{9/10}) / K_5^k.
\end{aligned}
\end{equation}
But $K_5^k -1 < 0$ and $\|x'\|_\infty < 1$ by construction, and so we want
\begin{multline}
\alpha >\|x'\|_\infty / (1 - (K_5^k -1 + \sqrt[n]{9/10})^{1/m} / K_5^{k/m} ) \geq \\
(1 -(K_5^k -1 + \sqrt[n]{9/10})^{1/m}/K_5^{k/m} )^{-1} = O(m/k)
\end{multline}
where we have used the expansion
\begin{equation}
K_5^{(1/m)} = O(1 + (K_5 -1)/m).
\end{equation}
Next define
\begin{equation}
K_7 = \min_{i}(\mathbb{P}(B^\mathsf{T}_{i}Y=0)),
\end{equation}
and note that we have
\begin{equation}
 \mathbb{P}(E_i^{2}) \leq  (1-K_{7})\cdot G( (B_{i})_{\min}/\alpha , \alpha-1 )
\end{equation}
by Proposition \ref{prop:singbound}.  So we get
\begin{equation}
 \mathbb{P}(E_i^2) \leq (1-K_{7})\cdot\exp((B_i)_{\min} - (B_i)_{\min}\log(\alpha)),
\end{equation}
and we require
\begin{multline}
\textstyle 1-\prod_{i=1}^{n}(1-\mathbb{P}(E_i^2))  \leq \\
-(1- (1-K_{7})\exp((B_i)_{\min} - (B_i)_{\min}\log(\alpha)))^{n} <  1/10.
\end{multline}
Since $\alpha = p\gamma$, for sufficiently large fixed $k''$ we have
\begin{multline}
(1-K_{6})\exp((B_i)_{\min} - (B_i)_{\min}\log(\alpha))  \geq \\
\exp(-k''\exp(-(B_i)_{\min}\cdot(\log(\alpha)-1))) 
\end{multline}
for a constant $K_6$.  This means:
\begin{multline}
\textstyle 1-\prod_{i=1}^{n}(1-\mathbb{P}(E_i^{2}))  \leq \\
1 - (\exp(-k''\exp(-(B_i)_{\min}(\log(\alpha)-1))))^{n} < 1/10 
\end{multline}
 holds when $p  > (K_8n)^{1/(B_i)_{\min}}$ for another constant $K_8$.  This choice implies $\mathbb{P}(\bigcup_{i=1}^{n}E_{i})< 1/5$ since $p\gamma=\alpha$.

We lastly examine $\mathbb{P}(E_{n+1})$.  Using Chebyshev's inequality and some algebra gives
\begin{multline}
\mathbb{P}(c^\mathsf{T}_{1}X + c^\mathsf{T}_{2}Y < (c^\mathsf{T}_{1}\hat{x} + c^\mathsf{T}_{2}\hat{y})/\alpha\beta) \leq \\
\mathbb{P}(|c^\mathsf{T}_{1}X + c^\mathsf{T}_{2}Y-(c^\mathsf{T}_{1}\hat{x} + c^\mathsf{T}_{2}\hat{y})/\gamma| > \sqrt{2}\cdot(c^\mathsf{T}_{1}\hat{x} + \\
c^\mathsf{T}_{2}\hat{y})/\sqrt{3}\gamma) \leq 3/2\cdot(c^\mathsf{T}_{1}\hat{x} + c^\mathsf{T}_{2}\hat{y})/\gamma < 3/10
\end{multline}
when we have $\alpha > \gamma$ and $(c^\mathsf{T}_{1}\hat{x} + c^\mathsf{T}_{2}\hat{y})/\alpha > 5$.
This implies
\begin{equation}
\mathbb{P}(E_{n+1}) = \mathbb{P}(c^\mathsf{T}_{1}X + c^\mathsf{T}_{2}Y < (c^\mathsf{T}_{1}\hat{x} + c^\mathsf{T}_{2}\hat{y})/\alpha\beta) < 2/5,
\end{equation}
which is the desired bound that was to be shown.
\end{proof}

We can now prove our first theorem, which follows by combining the above results.
\begin{theorem}
\label{thm:randround}
The parameters $\alpha,\beta,\gamma$ from Proposition \ref{prop:unbnd} are such that a feasible solution to GPIP generated by Algorithm \ref{alg:randound} is an $O(m/k + (n+d_1+d_2)^{1/([B, u, v])_{\min}})$-approximation.
\end{theorem}

\begin{proof}
If $(c^\mathsf{T}_{1}\hat{x} + c^\mathsf{T}_{2}\hat{y})/\alpha \leq 5$, then Algorithm \ref{alg:randound}  gives an $O((n+d_1+d_2)^{1/([B, u, v])_{\min}})$-approximation.  So we focus on the case $(c^\mathsf{T}_{1}\hat{x} + c^\mathsf{T}_{2}\hat{y})/\alpha > 5$.  Then by Proposition \ref{prop:unbnd} and the union bound we have:
\begin{multline}
\textstyle\mathbb{P}\Big(\bigcup_{i=1}^{n+1}E_i\bigcup_{i=1}^{d_1}Q_i\bigcup_{i=1}^{d_2}R_i\Big) \leq \mathbb{P}(\bigcup_{i=1}^nE_i) +\textstyle  \\
\textstyle\mathbb{P}(E_{n+1})+ \mathbb{P}(\bigcup_{i=1}^{d_1}Q_i) + \mathbb{P}(\bigcup_{i=1}^{d_2}R_i)< 1
\end{multline}
This means a feasible solution generated by Algorithm \ref{alg:randound} is an $O(m/k + (n+d_1+d_2)^{1/([B, u, v])_{\min}})$-approximation.
\end{proof}

\section{Deterministic Algorithm for GPIP}

We have constructed a randomized approximation algorithm for GPIP, but randomized algorithms are not guaranteed to produce a feasible solution \cite{raghavan1988probabilistic,srinivasan1999improved}.  In this section, we construct a deterministic approximation algorithm that always returns a feasible solution to GPIP.  Let $X$ be a vector of independent Bernoulli random variables, where $p_i$ is the probability $X_i$ equals one.  If $L$ is a set with $\mathbb{P}(X \in L) < 1$, then we can find an $x$ such that $x \notin L$ using  Algorithm \ref{alg:detound1} that uses a \emph{pessimistic estimator} to upper bound the probability of undesirable events.  Our algorithm is more aggressive than the approach from \cite{raghavan1988probabilistic,srinivasan1999improved}, which rounds in order to reduce the value of the pessimistic estimator; however, correctness of our algorithm follows from the same proof in \cite{raghavan1988probabilistic}.  

\begin{definition}[Pessimistic Estimator \cite{raghavan1988probabilistic}]
\label{def:pesest}
The function $U: [0,1]^\ell \rightarrow \mathbb{R}_+$ is a pessimistic estimator for set $L$ and Bernoulli random variables $(X_1,\ldots,X_\ell)$ with probability $(p_1,\ldots,p_\ell)$ of being one, if it satisfies the three properties:
\begin{enumerate}
\item $U(p_1,\ldots,p_\ell) < 1$;

\item $U(w_1,\ldots,w_i,p_{i+1}, \ldots, p_l) \geq$\\ $\min\{U(w_1,\ldots,w_i,k,p_{i+2},\ldots,p_\ell)\ |\ k\in\{0,1\}\}$;
\item $U(w_1,\ldots,w_i,p_{i+1},\ldots,p_\ell) \geq\mathbb{P}[X \in L | X_k = w$ for $k \in[i]]$, for all $i \in \{0,\ldots,\ell\}$ and $w \in \{0,1\}^{\ell}$.
\end{enumerate}
%
\end{definition}

\begin{algorithm}[t]
\begin{algorithmic}[1]
\Require Pessimistic Estimator $U$
\For {$i = 1,\ldots,\ell$}
\If {$U(x_1,\ldots,x_{i-1},1,p_{i+1},\ldots,p_\ell) < 1$}
\State choose $x_i = 1$
\Else
\State choose $x_i = 0$
\EndIf
\EndFor
\State \Return $x$
\end{algorithmic}
\caption{Deterministic Rounding Algorithm for Pessimistic Estimator \label{alg:detound1}}
\end{algorithm}

\subsection {Constructing a Pessimistic Estimator for GPIP}

To construct a deterministic approximation algorithm for GPIP, we need to build a pessimistic estimator for $\mathbb{P}\Big(\bigcup_{i=1}^{n+1}E_i\bigcup_{i=1}^{d_1}Q_i\bigcup_{i=1}^{d_2}R_i\Big) $.  Natural candidate functions are upper bounds to the probabilities of these events.  Before we provide these bounds, we define some notation: If $w$ is a binary vector, then $X(j) = w$ indicates we fix the first $j$ components of the vector $X$ to match $w$. We will use $u,Y$ similarly. With this convention, consider the following functions that are used as pessimistic estimators for each individual probability:
\begin{equation}
\begin{aligned}
h^{1}_{i}(j,w) &=  \mathbb{E}[(1+\delta^{1}_{i})^{U_i^\mathsf{T}X-\mu^{1}_{i}(1+\delta^{1}_{i})} | X(j) = w]\\
f^{1}_{i}(j,w) &=  \mathbb{E}[(1+\delta^{1}_{i})^{U_i^\mathsf{T}X-\mu^{1}_{i}(1+\delta^{1}_{i})} | X(j+1) = (w,0)]\\
g^{1}_{i}(j,w) &=  \mathbb{E}[(1+\delta^{1}_{i})^{U_i^\mathsf{T}X-\mu^{1}_{i}(1+\delta^{1}_{i})} | X(j+1) = (w,1)]
\end{aligned}
\end{equation}
and
\begin{equation}
\begin{aligned}
h^{2}_{i}(j,u) &=  \mathbb{E}[(1+\delta^{2}_{i})^{V^\mathsf{T}_{i}Y-\mu^{2}_{i}(1+\delta^{2}_{i})} | Y(j) = u]\\
f^{2}_{i}(j,u) &=  \mathbb{E}[(1+\delta^{2}_{i})^{V^\mathsf{T}_{i}Y-\mu^{2}_{i}(1+\delta^{2}_{i})} | Y(j+1) = (u,0)]\\
g^{2}_{i}(j,u) &=  \mathbb{E}[(1+\delta^{2}_{i})^{V^\mathsf{T}_{i}Y-\mu^{2}_{i}(1+\delta^{2}_{i})} | Y(j+1) = (u,1)]
\end{aligned}
\end{equation}
And we define the terms in (\ref{equ:floatedEquation}).  But the functions $h^{0},h^{1}, h^{2}, f^{1}, f^{2},g^{1},g^{2},f_{x}^{0},f_{y}^{0}, g_{x}^{0},g_{y}^{1}$ can be bigger than one, and so we define:
\begin{equation}
\begin{aligned}
& h'^{0} = \min\{ h^{0},1\}\\
&h'^{1}=\min\{h^{1},1\}&& h'^{2}=\min\{h^{2},1\} \\
& f'^{1}=\min\{f^{1},1\}&& f'^{2}=\min\{f^{2},1\}\\
& g'^{1}=\min\{g^{1},1\}&& g'^{2}=\min\{g^{2},1\} \\
& f'^{0}_{x}=\min\{f'^{0}_{x},1\}&& f'^{0}_{y}=\min\{f^{0}_{y},1\}\\
& g'^{0}_{x}=\min\{g_{x}^{0},1\}&\quad& g'^{0}_{y}=\min\{g^{0}_{y},1\}.
\end{aligned}
\end{equation}
With these definitions, we next construct (and prove its correctness) a pessimistic estimator in Theorem \ref{thm:pesest}.

 \begin{floatEq}
\begin{equation}
\begin{aligned}
&h^{0}_{i}(j,u,w,l) = \textstyle(\prod_{k=1}(1 - Y_{k}|Y(u) = l))\times\textstyle( \prod_{k=1}(1 -X_{k}|X(j) = w))+(1-\prod_{k=1}(1 - X_{k}|X(j) = w))\times\\&\hspace{3cm}\mathbb{E}[(1+\delta^{0}_{i})^{A_i^\mathsf{T}X-\mu^{0}_{i}(1+\delta^{0}_{i})} | X(j) = w]\\
&f^{0}_{xi}(j,u,w,l) = \textstyle(\prod_{k=1}(1 - Y_{k}|Y(u) = l))\times\textstyle( \prod_{k=1}(1 -X_{k}|X(j+1) = (w,0))) + (1-\prod_{k=1}(1 - Y_{k}|Y(u) = l))\times\\
&\hspace{3cm}\mathbb{E}[(1+\delta^{0}_{i})^{A^\mathsf{T}_{i}X-\mu^{0}_{i}(1+\delta^{0}_{i})} | X(j+1) = (w,0)]\\
&g^{0}_{xi}(j,u,w,l) \textstyle=(\prod_{k=1}(1 - Y_{k}|Y(u) = l))\times( \prod_{k=1}(1 -X_{k}|X(j+1) = (w,1))) + (1-\prod_{k=1}(1 - Y_{k}|Y(u) = l))\times\\
&\hspace{3cm}\mathbb{E}[(1+\delta^{0}_{i})^{A^\mathsf{T}_{i}X-\mu^{0}_{i}(1+\delta^{0}_{i})} | X(j+1) = (w,1)] \\
&f^{0}_{yi}(j,u,w,l) = \textstyle(\prod_{k=1}(1 - Y_{k}|Y(u+1) = (l,0)))( \prod_{k=1}(1 -X_{k}|X(j) = w)) + \\
&\textstyle\hspace{3cm}(1-\prod_{k=1}(1 - Y_{k}|Y(u+1) = (l,0)))\times\mathbb{E}[(1+\delta^{0}_{i})^{A^\mathsf{T}_{i}X-\mu^{0}_{i}(1+\delta^{0}_{i})} | X(j) = w]\\
&g^{0}_{yi}(j,u,w,l) = \textstyle(\prod_{k=1}(1 - Y_{k}|Y(u+1) = (l,1)))( \prod_{k=1}(1 -X_{k}|X(j) = w)) + \\
&\textstyle\hspace{3cm}(1-\prod_{k=1}(1 - Y_{k}|Y(u+1) = (l+1)))\times\mathbb{E}[(1+\delta^{0}_{i})^{A^\mathsf{T}_{i}X-\mu^{0}_{i}(1+\delta^{0}_{i})} | X(j) = w]\\
&h^{0}_{n+1}(j,u,w,l) = \mathbb{E}[(1-\delta_{n+1})^{(c^\mathsf{T}_{1}X + c^\mathsf{T}_{2}Y)-\mu_{n+1}(1-\delta_{n+1})} | X(j) = w, Y(u) = l]
\label{equ:floatedEquation}
\end{aligned}
\end{equation}
\end{floatEq}

\begin{theorem}
\label{thm:pesest}
The parameters $\alpha,\beta,\gamma$ from Proposition \ref{prop:unbnd} are such that
\begin{multline}
\label{eqn:pesest}
\textstyle 3-  \prod_{i=1}^{n}(1-h'^{0}_{i}(j,u,w,l)) - \prod_{i=1}^{d_1}(1-h'^{1}_{i}(j,w)) + \\
\textstyle-\prod_{i=1}^{d_2}(1-h'^{2}_{i}(u,l)) +h^{0}_{n+1}(j,u,w,l)
\end{multline}
is a pessimistic estimator for GPIP of the probability of $\mathbb{P}\Big(\bigcup_{i=1}^{n+1}E_i\bigcup_{i=1}^{d_1}Q_i\bigcup_{i=1}^{d_2}R_i\Big)$.  Furthermore, Algorithm \ref{alg:detound1} is an $O(m/k + (n+d_1+d_2)^{1/([B, u, v])_{\min}})$-approximation.
\end{theorem}

\begin{proof}
The proof follows the approach of \cite{srinivasan1999improved}, though one major difference is that we fix the order in which we round the variables. More specifically, we first round the Y variables, and then we round the X variables.  The first step is to prove some relations between the functions defined above:
\begin{equation}
\begin{aligned}
& 0\leq f'^{1}(j,w) \leq g'^{1}(j,w) \leq 1 \\
& (1-p_{j+1})f'^{1}(j,w) + p_{j+1}g'^{1}(j,w) \leq h'^{1}(j,w)
\end{aligned}
\end{equation}
These relations also hold for $h'^{2},f'^{2},g'^{2}$.  To see why they hold we omit the superscript and proceed to prove both relations.  Note that by fixing $f'(j,w)$ we see that it has the same value of $ g'(j,w)$ except for the element $x_{j+1}$ (or $y_{j+1}$) which has a nonnegative coefficient on either function. So by setting it equal to one we do not increase the function value. The first relation follows from this observation. To see the second relation, note that $h(j,w) = (1-p_{j})f(j,w) + p_{j}g(j,w)$ by definition of conditional expectation. But if $h_{i} < 1$ and $ g_{i} \leq 1$, then $f_i < 1$; so the second relation follows. If  $h_{i} < 1$ and $ g_{i} >1$, then $f_i < 1$; so the second relation follows as well. Lastly if $h_{i}\geq 1$, so $h'_{i}=g'_{i}=1$, and the second relation follows again. 

For $h'^{0}, f'^{0}_{x}, f'^{0}_{y},g'^{0}_{x},g'^{0}_{y}$ note that
\begin{multline}
h'^{0}(j,u,w,l) = (1-q_{j})\cdot f'^{0}_{y}(j,u,w,l) + \\q_{j}\cdot f'^{0}_{y}(j,u,w,l).
\end{multline}
But when we condition on $Y$, then the terms associated with $X$ remain constant. So the relationship holds directly from conditional probability.  Next we condition on $X$. Observe there are two cases.  The first case is all $B_i^\mathsf{T}Y$ are equal to zero: Then
\begin{equation}
h'^{0}(j,u) \geq (1-p_{j})f'^{0}_{x}(j',w,u) + p_{j}f'^{0}_{x}(j',w,u)
\end{equation}
since the only term present is $( \prod_{k=1}(1 -X_{k}|X(j) = w))$. So the expression follows from the definition of conditional probabilities.  The second case is at least one $Y$ is equal to one; this case is similar to $h'^{1}, f'^{1}, g'^{1}$. So we have
\begin{multline}
h'^{0}(j,u,w,l) \geq (1-p_{j})f'^{0}_{x}(j,u) + p_{j}f'^{0}_{x}(j,u) \\
= (1-q_{j})f'^{0}_{y}(j,u) + q_{j}f'^{0}_{y}(j,u)
\end{multline} 
and $0\leq f'^{0}_{x}(j,u,w,l) \leq g'^{0}{x}(j,u,w,l) \leq 1$ for all $(j,u)$.

Next we prove (\ref{eqn:pesest}) is a pessimistic estimator.  Our $\alpha,\beta,\gamma$ choice ensures the unconditional estimator is less than one, and that it upper bounds the failure probability. So the result follows if we prove the first two properties in Definition \ref{def:pesest}.  We show this by proving
\begin{multline}
U(w_1,\ldots,w_i,p_{i+1}, \ldots, p_\ell) \geq \\
p_{i+1}\cdot U(w_1,\ldots,w_i,0,p_{i+2},\ldots,p_\ell) \\
+ (1-p_{i+1})\cdot U(w_1,\ldots,w_i,1,p_{i+2},\ldots,p_\ell)
\end{multline} 
conditioned on $X$, and by proving
\begin{multline}
U(w_1,\ldots,w_i,p_{i+1}, \ldots, p_\ell) \geq \\
q_{i+1}\cdot U(w_1,\ldots,w_i,0,p_{i+2},\ldots,p_\ell) \\
+ (1-q_{i+1})\cdot U(w_1,\ldots,w_i,1,p_{i+2},\ldots,p_\ell) 
\end{multline} 
conditioned on $Y$. Let $p_{j}$ and $q_{j'}$ be the probability $X_{j} , Y_{j'}$ equal one, respectively. Then:
\begin{multline}
\mathbb{E}[(1-\delta_{n+1})^{\omega(X,Y)} | X(j) = w, Y(u) = \ell] = \\
(1-p_{j+1})\cdot\mathbb{E}[(1-\delta_{n+1})^{\omega(X,Y)} | X(j) = (w,0), Y(u) = \ell] + \\
p_{j+1}\cdot\mathbb{E}[(1-\delta_{n+1})^{\omega(X,Y)} | X(j) = (w,1), Y(u) = \ell] = \\
(1-q_{j+1})\cdot\mathbb{E}[(1-\delta_{n+1})^{\omega(X,Y)} | X(j) = w, Y(u,1) = \ell] + \\
q_{j+1}\cdot\mathbb{E}[(1-\delta_{n+1})^{\omega(X,Y)} | X(j) = w, Y(u,1) = \ell]  
\end{multline} 
where $\omega(X,Y) = (c^\mathsf{T}_{1}X + c^\mathsf{T}_{2}Y)-\mu_{n+1}(1-\delta_{n+1})$.


There are now two cases, where we either fix $X_{j+1}$ or fix $Y_{j+1}$.  The proofs are identical and so we consider the first case where we fix $X_{j+1}$, which gives that $\prod_{i=1}^{d_2}(1-h'^{2}_{i}(j,w,u))$ remains the same (since it only depends on the values of $Y$. Thus it is sufficient to show that
\begin{multline}
\textstyle\prod_{i=1}^{d_1}(1-h^{1}_{i}(j,w)) \leq \\
\textstyle(1-p_{j+1})\prod_{i=1}^{d_1}(1-f'^{1}_{i}(j,w)) + \\
\textstyle p_{j+1}\prod_{i=1}^{d_1}(1-g'^{1}_{i}(j,w))
\end{multline}
and
\begin{multline}
\textstyle\prod_{i=1}^{n}(1-h'^{0}_{i}(j,u,w,l)) \leq \\
\textstyle(1-p_{j+1})\prod_{i=1}^{n}(1-f'^{0}_{x,i}(j,u,w,l)) + \\
\textstyle p_{j+1}\prod_{i=1}^{b}(1-g'^{0}_{x,i}(j,u,w,l)).
\end{multline}
This holds trivially for $n=1$ by our choice of $\alpha,\gamma$.  Let $p_{j}= p$, and we omit the subscript for brevity.  Next we proceed by induction to show (51) (The proof for (50) is analogous and is therefore ommitted):  Assuming the above holds for $n-1$, we need to show that
\begin{multline}
\textstyle\prod_{i=1}^{n-1}(1-h'^{0}_{i}) \leq  \prod_{i=1}^{n-1}(1- (1-p)f'^{0}_{i} - pg'^{0}_{i} ) \leq \\
\textstyle\prod_{i=1}^{n-1}(1-p)f'^{0}_{i} +\prod_{i=1}^{n-1} pg'^{0}_{i} = \\
\textstyle(\prod_{i=1}^{n-1}(1-p)f'^{0}_{i} +\prod_{i=1}^{n-1} pg'^{0}_{i})(1-(1-p)f'^{0}_{n} - pg'^{0}_{n}) \leq \\
\textstyle(\prod_{i=1}^{n}(1-p)f'^{0}_{i} +\prod_{i=1}^{n} pg'^{0}_{i})
\end{multline}
Distributing and simplifying we get that
\begin{multline}
\textstyle p(1-p)(g'^{0}_{n} - f'^{0}_{n})(\prod_{i=1}^{n-1}(1-f'^{0}_{i}) + \\
\textstyle\prod_{i=1}^{n-1} (1-g'^{0}_{i})) \geq 0.
\end{multline}
The same holds for $h^{1}_{i}$. This shows
\begin{multline}
U(x(j),y) \geq(1-p_{j+1})U(X(j+1),Y| X_{j+1}=0) + \\
p_{j}U(X(j+1), Y| X_{j+1}=1)\
\end{multline}
and so
\begin{multline}
U(x(j),y) \geq \min\{U(X(j+1),Y | X_{j+1}=0), \\
U(X(j+1), Y| X_{j+1}=1)\}.
\end{multline}
This implies $U$ is a pessimistic estimator.
\end{proof}

\section{Computational Results}

This section describes the results of computational experiments in which Algorithm \ref{alg:detound1} was used to solve the optimization problem (\ref{eqn:dietmodel}) corresponding to our dietary planning formulation.  We used a database constructed from a subset of the Recipes Wikia \cite{recipes2016} consisting of about 2000 food recipes prepared from 130 raw ingredients. We conducted a series of experiments based on three databases sizes: small (20 recipes and 10 ingredients), medium (about 300 recipes and 50 ingredients), and large (the full database).  We also varied the horizon $N$ to be between one to ten weeks.

We conducted 100 repetitions where the food preferences $v,w$ in (\ref{eqn:dietmodel}) were randomly chosen, and Table \ref{table:r3} shows the average optimality gap -- with respect to the LP relaxation of (\ref{eqn:dietmodel}) -- of solutions computed using Algorithm \ref{alg:detound1}; standard deviation is in parenthesis.  Table \ref{table:r4} shows the computation time needed to calculate solutions using Algorithm \ref{alg:detound1}.  The average solution time is in seconds, and the standard deviation is in parenthesis.  Our experiments were conducted on a 2.2Ghz laptop computer with 8.00Gb of RAM and using Gurobi 7.0 \cite{gurobi} to compute the LP's for our algorithm.


\begin{table}[t]
\caption{Optimality Gap of Approximate Solutions} 
\centering 
\begin{tabular}{c c c c} 
\hline
\hline 
&\multicolumn{3}{c}{Instance Size}\\
Horizon ($N$) & Small & Medium & Large \\ 
\hline 
1 & 35\% (11\%) & 52\% (9\%) & 55\% (9\%) \\
3 & 36\% (\hphantom{0}7\%) & 48\% (5\%) & 57\% (5\%) \\
5 & 35\% (\hphantom{0}5\%) & 51\% (4\%) & 68\% (6\%) \\
7 & 35\% (\hphantom{0}5\%) & 53\% (4\%) & 71\% (5\%) \\
10 & 34\% (\hphantom{0}4\%) & 56\% (4\%) & 74\% (3\%) \\
\hline 
\hline
\end{tabular}
\label{table:r3}
\end{table}

\begin{table}[t]
\caption{Computation Time of Approximate Solutions} 
\centering 
\begin{tabular}{c c c c} 
\hline
\hline 
&\multicolumn{3}{c}{Instance Size}\\
Horizon ($N$) & Small & Medium & Large \\ 
\hline 
1 & 0.01 (0.01) & 0.02 (0.01) & 0.09 (0.02) \\
3 & 0.03 (0.01) & 0.14 (0.03) & 0.56 (0.08) \\
5 & 0.06 (0.01) & 0.27 (0.03) & 1.34 (0.14) \\
7 & 0.09 (0.01) & 0.56 (0.13) & 2.49 (0.15) \\
10 & 0.19 (0.01) & 0.95 (0.14) & 4.38 (0.47)\\
\hline 
\hline
\end{tabular}
\label{table:r4}
\end{table}

\section{Conclusion}

We gave a novel model formulation for dietary planning with temporal constraints, abstracted this formulation into a generalized packing integer program (GPIP), and constructed a deterministic approximation algorithm to solve GPIP.  Simulations with a real dietary database were used to evaluate our algorithm.  Interesting future directions include improving our algorithm by either tightening the bounds of the pessimistic estimators or by exploiting specific ordering properties in the rounding that occurs in our algorithm.

\addtolength{\textheight}{-12cm}   



\bibliographystyle{IEEEtran}
\bibliography{IEEEabrv,dynfood}

\end{document}